\documentclass[12pt]{amsart}

\theoremstyle{plain}
\newtheorem{thm}{Theorem}[section]
\newtheorem{cor}[thm]{Corollary}
\newtheorem{lem}[thm]{Lemma}

\theoremstyle{definition}
\newtheorem{rem}[thm]{Remark}
\newtheorem{defin}[thm]{Definition}

\newtheorem*{3loc}{Theorem 3.3}

\newtheorem*{conj}{Conjecture}

\def\cal{{\rm cal}}

\def\diam{{\rm diam}}

\def\dist{{\rm dist}}

\begin{document}

\title[Local sections of Serre fibrations with 3-manifold fibers]
{Local sections of Serre fibrations with 3-manifold fibers}

\author{N. Brodskiy}
\address{University of Tennessee, Knoxville, TN 37996, USA}
\email{brodskiy@@math.utk.edu}

\author{A. Chigogidze}
\address{University of North Carolina at Greensboro, Greensboro, NC
27402, USA} \email{chigogidze@uncg.edu}

\author{E.V.~Shchepin}
\address{Steklov Institute of Mathematics,
Russian Academy of Science, Moscow 117966, Russia}
\email{scepin@mi.ras.ru}

\thanks{The third author was partially supported by Russian Foundation
of Basic Research (project 08-01-00663).}


\begin{abstract}
It was proved by H. Whitney in 1933 that a Serre fibration of
compact metric spaces admits a global section provided every fiber
is homeomorphic to the unit interval [0,1]. An extension of the
Whitney's theorem to the case when all fibers are homeomorphic to
some fixed compact two-dimensional manifold was proved by the
authors~\cite{BCS}. The main result of this paper proves the existence of
local sections in a Serre fibration with all fibers homeomorphic
to some fixed compact three-dimensional manifold.
\end{abstract}

\maketitle

\section{Introduction}\label{S:intro}

The following problem is one of the central problems in geometric
topology~\cite{DS}. Let $p\colon E\to B$ be a Serre fibration of
separable metric spaces. Assume that the space $B$ is locally
$n$-connected and all fibers of $p$ are homeomorphic to some fixed
$n$-dimensional manifold $M^n$. Is $p$ a locally trivial
fibration?

In case $n=1$ an affirmative answer to this problem follows from
results of H.~Whitney~\cite{Whitney}.

\begin{conj}[Shchepin]
A Serre fibration with a locally arcwise connected metric base is
locally trivial if every fiber of this fibration is homeomorphic
to some fixed manifold $M^n$ of dimension $n \le 4$.
\end{conj}

In dimension $n=1$ the Shchepin's Conjecture is proved even for
non-compact fibers~\cite{RSS}. Shchepin proved that positive
solution of this Conjecture in dimension $n$ implies positive
solutions of both CE-problem and Homeomorphism Group problem in
dimension $n$~\cite{S, DS}. Since $CE$-problem was solved in a
negative way by A.N.~Dranishnikov, there are dimensional
restrictions in Shchepin's Conjecture.

Under the assumption of the base $B$ of the Serre fibration $p\colon E\to B$
being finite dimensional the Shchepin's Conjecture is proved in
dimensions $n=2$~\cite{HD} and $n=3$~\cite{H}.
An interesting result is obtained by S.~Ferry proving that $p$ is a Hurewicz fibration~\cite{F}.

The first step toward proving the Shchepin's Conjecture in
dimension $n=2$ over infinite dimensional base is made in~\cite{B, BCS} 
where existence of local and global
sections of the fibration is proved provided the base space is an $ANR$. 
The following theorem is the main result of this paper.

\begin{3loc}
Let $p\colon E\to B$ be a Serre fibration of $LC^2$-compacta with
all fibers homeomorphic to some fixed compact three-dimensional
manifold. If $B\in ANR$, then any section of $p$ over closed
subset $A\subset B$ can be extended to a section of $p$ over some
neighborhood of $A$.
\end{3loc}

Our strategy of constructing a section of a Serre fibration is as
follows (definitions are given in Section~\ref{S:Preliminaries}).
We consider the inverse (multivalued) mapping $p^{-1}$ and find
its compact submapping admitting continuous approximations. Then
we take very close continuous approximation and use it to find
again a compact submapping with small diameters of fibers
admitting continuous approximations. When we continue this process
we get a sequence of compact submappings with diameters of fibers
tending to zero. This sequence will converge to the desired
singlevalued submapping of $p^{-1}$ (section of $p$).

The major difference of the proof in this paper from the one in~\cite{BCS}
comes from the fact that any open subset of Euclidean plane is aspheric
(all homotopy groups vanish in dimensions $\ge 2$) which is far from being true in Euclidean 3-space.
In order to apply our technique to 3-dimensional
manifolds we introduce a new property called hereditarily
coconnected asphericity.

\section{Preliminaries on spaces and multivalued mappings}\label{S:Preliminaries}

Let us recall some definitions and introduce our notations. All
spaces will be separable metrizable. We equip the product $X\times
Y$ with the metric
$$\dist_{X\times Y}((x,y),(x',y'))=\dist_X(x,x')+\dist_Y(y,y').$$
By $O(x,\varepsilon)$ we denote the open
$\varepsilon$-neighborhood of the point $x$.

A multivalued mapping $F\colon A\to Y$ is called {\it submapping}
(or {\it selection}) of multivalued mapping $G\colon X\to Y$ if $A$ is a subspace of $X$ and
$F(x)\subset G(x)$ for every $x\in A$. The {\it gauge} of a
multivalued mapping $F\colon X\to Y$ is defined as $\cal
(F)=\sup\{\diam F(x)\mid x\in X\}$. The {\it graph} of multivalued
mapping $F\colon X\to Y$ is the subset $\Gamma_F=\{(x,y)\in
X\times Y\mid y\in F(x)\}$ of the product $X\times Y$. For
arbitrary subset $\mathcal U\subset X\times Y$ denote by $\mathcal
U(x)$ the subset $pr_Y(\mathcal U\cap(\{x\}\times Y))$ of $Y$.
Then for the graph $\Gamma_F$ we have $\Gamma_F(x)=F(x)$.

A multivalued mapping $G\colon X\to Y$ is called {\it complete} if
there exists a $G_\delta$-set $S\subset X\times Y$ containing the
graph $\Gamma_G$ such that all sets $\{x\}\times G(x)$ are closed
in $S$. Notice that any compact-valued mapping is complete. A
multivalued mapping $F\colon X\to Y$ is called {\it upper
semicontinuous} if for any open set $U\subset Y$ the set $\{x\in
X\mid F(x)\subset U\}$ is open in $X$. A {\it compact} mapping is
an upper semicontinuous multivalued mapping with compact images of
points.

Let $Z$ be a space. A sequence $\{Z_k\}$ (finite or infinite) of
subspaces
$$ Z_0\subset Z_1\subset Z_2\subset\dots\subset Z $$
is called a {\it  filtration} of $Z$.
The number of elements of the filtration is called {\it length} of the filtration.
Given a filtration $\{Z_k\}$ of $Z$, its {\it subfiltration} $\{Z'_k\}$ is a filtration of $Z$ such that $Z'_k\subset Z_k$ for every $k$.
A sequence of multivalued
mappings $\{F_k\colon X\to Y\}$ is called a {\it filtration of
multivalued mapping} $F\colon X\to Y$ if for any $x\in X$ the
sequence $\{F_k(x)\}$ is a filtration of $F(x)$. We say that a
filtration of multivalued mappings $G_i\colon X\to Y$ is {\it
compact} if every mapping $G_i$ is compact.

A pair of spaces $V\subset U$ is called {\it $n$-aspheric} if
every continuous mapping of the $n$-sphere into $V$ is homotopic
to a constant mapping in $U$. We assume that {\it $-1$-asphericity}
of the pair $V\subset U$ means exactly that $U$ is non-empty.
A pair of compacta $K\subset K'$ is
called {\it approximately $n$-aspheric} if for some embedding of
$K'$ into $ANR$-space for any neighborhood $U$ of the set $K'$
there is a neighborhood $V$ of the set $K$ such that the pair
$V\subset U$ is $n$-aspheric. A compact space $K$ is called {\it
approximately aspheric} if the pair $K\subset K$ is approximately
$n$-aspheric for every $n\ge 2$.

A pair of spaces $V\subset U$ is called {\it polyhedrally
$n$-connected} if for any finite $n$-dimensional polyhedron $M$
and its closed subpolyhedron $A$ any mapping of $A$ in $V$ can be
extended to a map of $M$ into $U$.

\begin{defin}
We say that a subset $A$ of a space $Z$ is {\it coconnected} if
the complement $Z\setminus A$ is connected.

A pair $V\subset U$ of proper subsets of a space $Z$ is called
{\it coconnected} if $Z\setminus U$ lies in one connected
component of $Z\setminus V$.
\end{defin}

\begin{defin}
We call a space $Z$ {\it hereditarily coconnectedly aspheric}
if any non-separating compactum $K\subset Z$ is approximately aspheric.

A space $Z$ is said to be {\it locally hereditarily coconnectedly aspheric}
if any point $z\in Z$ has a hereditarily coconnectedly aspheric neighborhood.
\end{defin}

\begin{rem}\label{remark hereditarily coconnectedly aspheric lemma}
An important example of hereditarily coconnectedly aspheric space
is Euclidean 3-space~\cite[Lemma~2.4]{DS}. Therefore, any
3-dimensional manifold is locally hereditarily coconnectedly
aspheric.
\end{rem}

Now we consider different properties of pairs of spaces and define the
corresponding local properties for spaces and multivalued maps.
We follow definitions and notations from~\cite{DM}.

\begin{defin}
An ordering $\alpha$ of the subsets of a space $Y$ is {\it proper} provided:
\begin{itemize}
\item[(a)] If $W\alpha V$, then $W\subset V$;
\item[(b)] If $W\subset V$, and $V\alpha R$, then $W\alpha R$;
\item[(c)] If $W\alpha V$, and $V\subset R$, then $W\alpha R$.
\end{itemize}
\end{defin}

We are going to use the following proper orderings:
\begin{itemize}
\item[(1)] $V\alpha U$ means $U$ is non-empty and $V\subset U$;

\item[(2)] $V\alpha U$ means the pair $V\subset U$ is $k$-aspheric
for every $k\le n$;

\item[(3)] $V\alpha U$ means the pair $V\subset U$ is polyhedrally
$n$-connected;

\item[(4)] $V\alpha U$ means the pair $V\subset U$ is coconnected;

\item[(5)] $V\alpha U$ means $V$ is hereditarily coconnectedly
aspheric.
\end{itemize}

\begin{defin}
Let $\alpha$ be a proper ordering. A space $Y$ is {\it locally of
type $\alpha$} if, whenever $y\in Y$ and $V$ is a neighbourhood of
$y$, then there a neighbourhood $W$ of $y$ such that $W\alpha V$.
\end{defin}

Let us introduce terminology for spaces which are locally of type
$\alpha$ for the examples of proper orderings $\alpha$ described
above:

\begin{itemize}
\item[(1)] any space is locally of type $\alpha$;

\item[(2)] $X$ is locally of type $\alpha$ means $X$ is {\it
locally $n$-connected} (notation: $X\in LC^n$);

\item[(3)] $X$ is locally of type $\alpha$ means $X$ is {\it
locally polyhedrally $n$-connected};

\item[(4)] $X$ is locally of type $\alpha$ means $X$ is {\it
locally coconnected};

\item[(5)] $X$ is locally of type $\alpha$ means $X$ is {\it
locally hereditarily coconnectedly aspheric}.
\end{itemize}

We use the word "equi" for local properties of multivalued maps.

\begin{defin}
Let $\alpha$ be a proper ordering. A multivalued mapping $F\colon
X\to Y$ is {\it equi locally of type $\alpha$} if for any points
$x\in X$ and $y\in F(x)$ and for any neighbourhood $V$ of $y$ in
$Y$ there exist neighbourhoods $W$ of $y$ in $Y$ and $U$ of $x$ in
$X$ such that $(W\cap F(x'))\alpha(V\cap F(x'))$ provided $x'\in
U$.
\end{defin}

Let us introduce terminology for multivalued mappings which are
equi locally of type $\alpha$ for the examples of proper orderings
$\alpha$ described above:

\begin{itemize}
\item[(1)] $F$ is equi locally of type $\alpha$ means $F$ is {\it
lower semicontinuous};

\item[(2)] $F$ is locally of type $\alpha$ means $F$ is {\it equi
locally $n$-connected} (briefly, $F$ is equi-$LC^n$);

\item[(3)] $F$ is locally of type $\alpha$ means $F$ is {\it equi
locally polyhedrally $n$-connected};

\item[(4)] $F$ is locally of type $\alpha$ means $F$ is {\it equi
locally coconnected};

\item[(5)] $F$ is locally of type $\alpha$ means $F$ is {\it equi
locally hereditarily coconnectedly aspheric}.
\end{itemize}

The following Lemma is easy to prove~\cite[Lemma 2.7]{BCS}.

\begin{lem} \label{lempolcont}
Any equi-$LC^n$ multivalued mapping is equi locally polyhedrally
$n$-continuous.
\end{lem}

\begin{cor} \label{lempolconnect}
Any $LC^n$ space $X$ is locally polyhedrally $n$-connected.
\end{cor}

\begin{proof}
Consider a multivalued mapping from the one-point space onto $X$
and apply Lemma~\ref{lempolcont}.
\end{proof}

The following lemma is easy to prove. We
will use it with different properties $\alpha$ in
Section~\ref{34}.

\begin{lem}\label{L32}
Let $\alpha$ be a proper ordering. Let $F\colon X\to Y$ be a
multivalued mapping of compactum $X$ to a metric space $Y$.
Suppose that $F$ is equi locally of type $\alpha$ and contains a
compact submapping $\Psi$. Then for any $\varepsilon>0$ there
exists a positive number $\delta$ such that for every point
$(x,y)\in O(\Gamma_\Psi,\delta)$ we have $(O(y,\delta)\cap
F(x))\alpha (O(y,\varepsilon)\cap F(x))$.
\end{lem}

A filtration $\{F_i\}$ of multivalued maps is called {\it equi
locally connected} if for any $i$ the mapping $F_i$ is equi
locally $i$-connected. A filtration of multivalued maps $\{F_i\}$
is called {\it polyhedrally connected} if every pair
$F_{i-1}(x)\subset F_{i}(x)$ is polyhedrally $i$-connected. A
filtration of compact mappings $\{F_m\colon X\to Y\}$ is called
{\it approximately connected} if for any point $x\in X$ and for
any $k$ the pair $F_k(x)\subset F_{k+1}(x)$ is approximately
$k$-aspheric.

The following Lemma is a weak form of Compact Filtration Lemma
from~\cite{SB}.

\begin{lem}\label{L31}
Any polyhedrally connected equi locally connected finite
filtration of complete mappings of a compact space contains a
compact approximately connected subfiltration of the same length.
\end{lem}

\begin{defin}
A multivalued mapping $F\colon X\to Y$ {\it admits continuous
approximations} if every neighborhood of the graph $\Gamma_F$ in
$X\times Y$ contains a graph of some single-valued continuous map
$f\colon X\to Y$.
\end{defin}

\begin{thm} \label{thmapprox}
Suppose that a compact multivalued mapping of separable metric
ANRs $F\colon X\to Y$ admits a compact approximately connected
filtration of infinite length. Then for any compact space
$K\subset X$ every neighborhood of the graph $\Gamma_F(K)$
contains the graph of a single-valued and continuous mapping
$f\colon K\to Y$.
\end{thm}

If a pair $G_0\subset G_1$ of proper subsets of a space $Z$ is
coconnected, then we can define an operation of {\it
$G_1$-coconnectification} on subsets of $G_0$ as follows: for a
subset $F_0\subset G_0$ its $G_1$-coconnectification is the union
of $F_0$ and all components of $Z\setminus F_0$ which do not
intersect $Z\setminus G_1$. Clearly, the $G_1$-coconnectification
of $F_0$ is the minimal subset $F_1\subset G_1$ containing $F_0$
such that $Z\setminus F_1$ is connected and contains $Z\setminus
G_1$.

\begin{lem} \label{lemma closedness of coconnectification}
Suppose that $Z$ is a locally 0-connected space and $G_0\subset
G_1$ is a coconnected pair of proper subsets of $Z$. Let $F_0$ be
a subset of $G_0$. If $F_0$ is closed in $Z$ then the
$G_1$-coconnectification $F_1$ of $F_0$ is also closed in $Z$.
\end{lem}

\begin{proof}
Suppose that $\{z_n\}$ is a sequence of points in $F_1$ converging
to a point $z\in Z\setminus F_1$. If infinitely many points $z_n$
belong to $F_0$, then $z\in F_0$ by closedness of $F_0$. So, we
assume that $z_n \notin F_0$ for all $n$. Then for any $n$ the
points $z$ and $z_n$ belong to different connected components of
$Z\setminus F_0$. Use local 0-connectedness of $Z$ to find a path
$w_n$ from $z$ to $z_n$ such that the diameter of $w_n$ tends to
$0$ as $n\to \infty$. Since each path $w_n$ must intersect $F_0$,
the point $z$ is a limit point of $F_0$. Contradiction.
\end{proof}

If a multivalued mapping $F\colon X\to Y$ contains proper
submappings $G_0$ and $G_1$ such that for any $x\in X$ the pair
$G_0(x)\subset G_1(x)$ is coconnected in $F(x)$, then for any
submapping $F_0\subset G_0$ we define a {\it
$G_1$-coconnectification} of $F_0$ as a multivalued mapping taking
a point $x\in X$ to the $G_1(x)$-coconnectification of $F_0(x)$.

\begin{lem} \label{lemmacocon}
Suppose that equi-$LC^0$ multivalued mapping $F\colon X\to Y$
contains proper submappings $G_0\subset G_1$ such that $G_1$ is
compact and for any $x\in X$ the pair $G_0(x)\subset G_1(x)$ is
coconnected in $F(x)$. Then for any compact submapping $F_0\subset
G_0$ its $G_1$-coconnectification $F_1$ is a compact submapping of
$G_1$.
\end{lem}

\begin{proof}
Since the $G_1(x)$-coconnectification of the set $F_0(x)$ is
closed in $F(x)$ by Lemma~\ref{lemma closedness of
coconnectification} and is contained in $G_1(x)$, the set $F_1(x)$
is compact.

Let us prove that $F_1$ is upper semicontinuous. Suppose to the
contrary that for some point $y\in F(x)\setminus F_1(x)$ there is
a sequence $\{y_i\}_{i=1}^\infty$ of points converging to $y$ such
that $y_i$ belongs to a set $F_1(x_i)\setminus F_0(x_i)$ for some
$x_i\in X$. Fix a point $z\in F(x)\setminus G_1(x)$. The points
$y$ and $z$ belong to connected set $F(x)\setminus F_1(x)$ which
is open in $F(x)$ and therefore is locally path connected (since
$F(x)\in LC^0$). Hence, there exists a path $s\colon [0,1]\to
F(x)\setminus F_1(x)$ such that $s(0)=y$ and $s(1)=z$. Since $F$
is lower semicontinuous and $G_1$ is upper semicontinuous, there
is a sequence of points $\{z_i\in F(x_i)\setminus
G_1(x_i)\}_{i=M}^\infty$, converging to $z$.

Using equi-$LC^0$ property of the mapping $F$ we can choose a
sequence of maps $\{s_i\colon [0,1]\to F(x_i)\}_{i=M'}^\infty$
such that $s_i(0)=y_i$, $s_i(1)=z_i$ and the paths $s_i$ converge
to the path $s$ uniformly (this follows from general results on
continuous selections~\cite{Mi}, although an elementary proof
exists which is straightforward but too technical). Since the path
$s$ does not intersect $F_0(x)$ and $F_0$ is upper semicontinuous,
for all but finitely many $i$ the path $s_i$ does not intersect
$F_0(x_i)$. It means that the points $y_i$ and $z_i$ belong to the
same connected component of the set $F(x_i)\setminus F_0(x_i)$,
which contradicts to the choices of $y_i$ and $z_i$.
\end{proof}

\begin{defin}
The mapping $f \colon X \to Y$ is said to be {\it topologically regular}
provided that if $\varepsilon > 0$ and $y \in Y$,
then there is a positive number $\delta$ such that
$dist(y,y') < \delta$,  $y' \in Y$,  implies that
there is a homeomorphism of $f^{-1}(y)$ onto $f^{-1}(y')$
which moves no point as much as $\varepsilon$
(i.e. an $\varepsilon$-homeomorphism).
\end{defin}

Note that since the Poincare Conjecture is true, any Serre fibration of $LC^2$-compacta
with all fibers homeomorphic to some fixed compact three-dimensional manifold is topologically regular~\cite{H}.

\begin{lem} \label{lemmafibrprop}
If $p\colon E\to B$ is a topologically regular mapping of compacta
with all fibers homeomorphic to some fixed compact three-dimensional manifold, then
the multivalued mapping $p^{-1}\colon B\to E$ is
\begin{itemize}
\item equi locally hereditarily coconnectedly aspheric

\item equi locally coconnected

\item equi locally polyhedrally $2$-connected.
\end{itemize}
\end{lem}

\begin{proof}
Fix a point $q\in E$ and $\varepsilon>0$. We will find $\delta>0$
such that for any point $x\in p(O(q,\delta))$ there exist subsets
$D^3$ and $O^3$ of the fiber $p^{-1}(x)$ such that
\[  O(q,\delta)\cap p^{-1}(x)\subset D^3\subset O^3\subset
    O(q,\varepsilon)\cap p^{-1}(x)  \]
where $D^3$ is homeomorphic to closed 3-ball and $O^3$ is
homeomorphic to $\mathbb R^3$. Then the first property of $p^{-1}$
follows from the fact that $O^3$ is hereditarily coconnectedly
aspheric (see Remark~\ref{remark hereditarily coconnectedly aspheric lemma}).
The last two properties follow
from coconnectedness of the pair $D^3\subset O^3$ and
contractibility of $D^3$ respectively.

Take a neighborhood $O^3_q$ of the point $q$ in the fiber $p^{-1}(p(q))$ such
that $O^3_q$ is homeomorphic to $\mathbb R^3$ and is contained in $O(q,\varepsilon/2)$.
Note that if $h$ is $\varepsilon/2$-homeomorphism of $O^3_q$, then
$h(O^3_q)$ is contained in $O(q,\varepsilon)$.
Let $D^3_q$ be a neighborhood of $q$ in $p^{-1}(p(q))$
homeomorphic to closed 3-ball.
Take a number $\sigma>0$ such that $O(q,\sigma)\cap p^{-1}(p(q))$
is contained in $D^3_q$.
Choose a positive number $\delta<\sigma/2$ such that for any point
$x\in O(p(q),\delta)$ there exists $\sigma/2$-homeomorphism
of the fiber $p^{-1}(p(q))$ onto $p^{-1}(x)$.
Now take a point $x\in p(O(q,\delta))$ and fix $\sigma/2$-homeomorphism $h$
of the fiber $p^{-1}(p(q))$ onto $p^{-1}(x)$.
By the choice of $\sigma$, the set $h(D^3_q)$ contains
$O(q,\sigma/2)\cap p^{-1}(x)$. Therefore, we have
\[  O(q,\delta)\cap p^{-1}(x)\subset h(D^3_q)\subset h(O^3_q)\subset
    O(q,\varepsilon)\cap p^{-1}(x).  \]
\end{proof}

\section{Fibrations with 3-manifold fibers}\label{34}

\begin{lem} \label{lemma3shrinking}
Let $F\colon X\to Y$ be equi locally hereditarily coconnectedly
aspheric, equi locally coconnected, equi-$LC^1$ compact-valued
mapping of $ANR$-space $X$ to Banach space $Y$. Suppose that a
compact submapping $\Psi\colon A\to Y$ of $F|_A$ is defined on a
compactum $A\subset X$ and admits continuous approximations. Then
for any $\varepsilon>0$ there exists a neighborhood $OA$ of $A$
and a compact submapping $\Psi'\colon OA\to Y$ of $F|_{OA}$ such
that $\Gamma_{\Psi'}\subset O(\Gamma_\Psi,\varepsilon)$, $\Psi'$
admits a compact approximately connected filtration of infinite
length, and $\cal \Psi'<\varepsilon$.
\end{lem}

\begin{proof}
Fix a positive number $\varepsilon$. Apply Lemma~\ref{L32} with
$\alpha$ being equi local hereditary coconnected asphericity to
get a positive number $\varepsilon_3<\varepsilon$. Apply
Lemma~\ref{L32} with $\alpha$ being lower coconnectedness to get a
positive number $\varepsilon_2<\varepsilon_3/2$. By
Lemma~\ref{lempolcont} the mapping $F$ is lower polyhedrally
2-continuous. Subsequently applying Lemma~\ref{L32} with $\alpha$
being polyhedral $n$-continuity for $n=2,1,0$, we find positive
numbers $\varepsilon_1$, $\varepsilon_0$, and $\delta$ such that
$\delta<\varepsilon_0<\varepsilon_1<\varepsilon_2$ and for every
point $(x,y)\in O(\Gamma_{\Psi},\delta)$ the pair
$(O(y,\varepsilon_1)\cap F(x),O(y,\varepsilon_2)\cap F(x))$ is
polyhedrally 2-connected, the pair $(O(y,\varepsilon_0)\cap
F(x),O(y,\varepsilon_1)\cap F(x))$ is polyhedrally 1-connected,
and the intersection $O(y,\varepsilon_0)\cap F(x)$ is not empty.

Let $f\colon K\to Y$ be a continuous single-valued mapping
whose graph is contained in $O(\Gamma_{\Psi},\delta)$.
Let $f'\colon\mathcal OK\to Y$ be a continuous extension of the mapping $f$
over some neighborhood $\mathcal OK$ such that the graph of $f'$
is contained in $O(\Gamma_{\Psi},\delta)$.
Now we can define a polyhedrally connected filtration
$G_0\subset G_1\subset G_2\colon \mathcal OK\to Y$
of the mapping $F|_{\mathcal OK}$ by the equality
\[
G_i(x)=O(f'(x),\varepsilon_i)\cap F(x).
\]
Since the set $\cup_{x\in\mathcal OK}\{\{x\}\times
O(f'(x),\varepsilon_i)\}$ is open in the product $\mathcal
OK\times Y$ and the mapping $F$ is complete, then $G_i$ is also
complete. Clearly, the set $G_2^\Gamma(x)$ is contained in
$O(\Gamma_\Psi,2\varepsilon_2)$. Now, applying Lemma~\ref{L31} to
the filtration $G_0\subset G_1\subset G_2$, we obtain a compact
approximately connected subfiltration $F_0\subset F_1\subset
F_2\colon \mathcal OK\to Y$. By the choice of $\varepsilon_2$ we
can find upper semicontinuous closed-valued mapping $F_3$
containing $F_2$ such that the pair $F_2(x)\subset F_3(x)$ is
coconnected in $F(x)$ for any $x\in X$. Since the map $F$ is
compact-valued, $F_3$ is compact. Using Lemma~\ref{lemmacocon} and
the choice of $\varepsilon_3$ we find a coconnectification $F_4$
of $F_2$ inside $F_3$. Then $F_4$ is compact submapping of $F$
having approximately aspheric point-images. Therefore, the
infinite filtration $F_0\subset F_1\subset F_4\subset F_4\subset
F_4\ldots$ is approximately connected and we can put $\Psi'=F_4$.
\end{proof}

\begin{thm} \label{thm3approxsections}
Let $F\colon X\to Y$ be equi locally hereditarily coconnectedly
aspheric, equi locally coconnected, equi-$LC^1$ compact-valued
mapping of locally compact $ANR$-space $X$ to Banach space $Y$.
Suppose that a compact submapping $\Psi\colon A\to Y$ of $F|_A$ is
defined on a compactum $A\subset X$ and admits continuous
approximations. Then for any $\varepsilon>0$ there exists a
neighborhood $OA$ of $A$ and a single-valued continuous selection
$s\colon OA\to Y$ of $F|_{OA}$ such that $\Gamma_s\subset
O(\Gamma_\Psi,\varepsilon)$.
\end{thm}

\begin{proof}
Fix $\varepsilon>0$. By Lemma~\ref{lemma3shrinking} there is a
neighborhood $U_1$ of $A$ in $X$ and a compact submapping
$\Psi_1\colon U_1\to Y$ of $F|_{U_1}$ such that
$\Gamma_{\Psi_1}\subset O(\Gamma_\Psi,\varepsilon)$, $\Psi_1$
admits a compact approximately connected filtration of infinite
length, and $\cal \Psi_1<\varepsilon$. Since $X$ is locally
compact and $A$ is compact, there exists a compact neighborhood
$OA$ of $A$ such that $OA\subset U_1$. By Theorem~\ref{thmapprox}
the mapping $\Psi_1|_{OA}$ admits continuous approximations. Take
$\varepsilon_1<\varepsilon$ such that the neighborhood $\mathcal
U_1=O(\Gamma_{\Psi_1}(OA),\varepsilon_1)$ lies in
$O(\Gamma_\Psi,\varepsilon)$.

Now by induction with the use of Lemma~\ref{lemma3shrinking}, we
construct a sequence of neighborhoods $U_1\supset U_2\supset
U_3\supset\dots$ of the compactum $OA$, a sequence of compact
submappings $\{\Psi_k\colon U_k\to Y\}^\infty_{k=1}$ of the
mapping $F$, and a sequence of neighborhoods $\mathcal
U_k=O(\Gamma_{\Psi_1}(OA),\varepsilon_k)$ such that for every
$k\ge 2$ we have $\cal
\Psi_k<\varepsilon_{k-1}/2<\varepsilon/2^k$, and $\mathcal
U_k(OA)$ is contained in $\mathcal U_{k-1}(OA)$. It is not
difficult to choose the neighborhood $\mathcal U_k$ of the graph
$\Gamma_{\Psi_k}$ in such a way that for every point $x\in U_k$
the set $\mathcal U_k(x)$ has diameter less than $\frac{3}{2^k}$.

Then for any $m\ge k\ge 1$ and for any point $x\in OA$ we have
$\Psi_m(x)\subset O(\Psi_k(x),\frac{3}{2^k})$; this implies the
fact that the sequence of compacta $\{\Psi_k(x)\}_{k=1}^\infty$ is Cauchy
(in Hausdorff metric).
Since $\Psi_k(x)\subset F(x)$ for all $k$ and every $x\in OA$, there exists a limit $s(x)\in F(x)$ of
this sequence (recall that $F(x)$ is compact). The mapping $s\colon OA\to Y$ is single-valued by the condition
$\cal\Psi_k<\frac{1}{2^k}$ and is upper semicontinuous (and,
therefore, is continuous) by the upper semicontinuity of all the
mappings $\Psi_k$. Thus $s$ is a selection of the mapping $F$.
\end{proof}

\begin{thm}
Let $p\colon E\to B$ be a Serre fibration of $LC^2$-compacta with
all fibers homeomorphic to some fixed compact three-dimensional
manifold. If $B\in ANR$, then any section of $p$ over closed
subset $A\subset B$ can be extended to a section of $p$ over some
neighborhood of $A$.
\end{thm}

\begin{proof}
Note that the Serre fibration $p$ is topologically
regular~\cite{H}.

Let $s\colon A\to E$ be a section of $p$ over $A$. Embed $E$ into
Hilbert space $l_2$ and consider a compact-valued mapping $F\colon
B\to l_2$ defined as follows:
$$
   F(b)=\begin{cases} s(b), &\text{if \;$b\in A$}\\
               p^{-1}(b),   &\text{if \;$x\in B\setminus A$.}\end{cases}
$$
It follows from Lemma~\ref{lemmafibrprop} that the mapping $F$ is
equi locally hereditarily coconnectedly aspheric, equi locally
coconnected, and equi-$LC^1$. We can apply
Theorem~\ref{thm3approxsections} to the mapping $F$ and its
submapping $s$ to find a single-valued continuous selection
$\widetilde s\colon OA\to l_2$ of $F|_{OA}$. Since the restriction
$F|_{A}$ is single-valued and equal to $s$, we have $\widetilde
s|_{A}=s$. Clearly, $\widetilde s$ defines a section of the
fibration $p$ over $OA$ extending $s$.
\end{proof}

\end{document}